\documentclass{article}
\usepackage{amsmath,amsfonts,amsthm,amssymb,amscd,color,xcolor,mathrsfs,eufrak, mathtools,mathrsfs,enumitem}
\usepackage{microtype}
\usepackage{hyperref}
\usepackage{comment}

\colorlet{darkblue}{blue!50!black}
\colorlet{darkmagenta}{magenta!80!black}

\hypersetup{
    colorlinks,%
    citecolor=darkblue,%
    filecolor=red,%
    linkcolor=darkblue,%
    urlcolor=blue,%
    pdfnewwindow=true,%
    pdfstartview={FitH}
}

\newcommand{\p}{\partial}
\newcommand{\e}{\varepsilon}

\newcommand{\R}{{\mathbb R}}
\newcommand{\Z}{{\mathbb Z}}

\newcommand{\LLL}{{\mathscr L}}

\newcommand{\XX}{{\cal X}}

\newcommand{\dd}{{\textup d}}

\newcommand{\aaaa}{{\mathfrak a}}

\newcommand{\diver}{\mathop{\rm div}\nolimits}

\newcommand{\sgn}{\mathop{\rm sgn}\nolimits}

\theoremstyle{plain}
\newtheorem*{mtheorem}{Main Theorem}

\newtheorem{theorem}{Theorem}[section]
\newtheorem{lemma}[theorem]{Lemma}
\newtheorem{proposition}[theorem]{Proposition}
\newtheorem{corollary}[theorem]{Corollary}
\theoremstyle{definition}

\newtheorem{op}{Open problem}
\theoremstyle{remark}
\newtheorem{remark}[theorem]{Remark}

\numberwithin{equation}{section}

\makeatletter
\let\@fnsymbol\@arabic
\makeatother

\begin{document}
\title{Stabilisation of a viscous conservation law\\ 
by a one-dimensional external force}
\author{Ana Djurdjevac\footnote{Freie Universit\"at Berlin, Arnimallee 9, 14195 Berlin, Germany; e-mail: \href{mailto:adjurdjevac@zedat.fu-berlin.de}{adjurdjevac@zedat.fu-berlin.de}.} \and Armen Shirikyan\footnote{Department of Mathematics, CY Cergy Paris University, CNRS UMR 8088, 2 avenue Adolphe Chauvin, 95302  Cergy--Pontoise, France; e-mail: \href{mailto:Armen.Shirikyan@cyu.fr}{Armen.Shirikyan@cyu.fr}.}}
\date{}
\maketitle

\vspace{-0.45cm}
\begin{abstract}
We study a damped scalar conservation law driven by the sum of a fixed external force and a localised one-dimensional control. The problem is considered in a bounded domain and is supplemented with the Dirichlet boundary condition. It is proved that any solution of the uncontrolled equation can be exponentially stabilised. As a consequence, we obtain the global approximate controllability to trajectories by a one-dimensional localised control. 


\end{abstract}

\tableofcontents
\setcounter{section}{-1}

\section{Introduction}
This paper deals with the problem of stabilisation of a damped-driven conservation law in arbitrary space dimension. To simplify the presentation, we confine ourselves in the introduction to the one-dimensional case. We thus fix an interval $I:=[a,b]$ and consider Burgers' equation
\begin{equation} \label{burgers}
	\p_tu+\p_xA(u)-\nu\p_x^2u=h(t,x), \quad x\in I,
\end{equation}
supplemented with Dirichlet's boundary condition
\begin{equation} \label{dirichlet}
	u(t,a)=u(t,b)=0
\end{equation}
and the initial condition
\begin{equation} \label{IC-1D}
	u(0,x)=u_0(x).
\end{equation} 
Here $\nu>0$ is a parameter, $A:\R\to\R$ is a given continuously differentiable function, and~$h$ is an external force. It will be convenient to study~\eqref{burgers}--\eqref{IC-1D} in H\"older's classes of functions, so we fix a number $\gamma\in(0,1)$ and denote by~$C^\gamma(I)$ the space of continuous functions $g:I\to\R$ such that 
$$
\|g\|_{C^\gamma}:=\sup_{x\in I}|g(x)|+\sup_{x_1,x_2\in I}\frac{|g(x_1)-g(x_2)|}{|x_1-x_2|^\gamma}<\infty. 
$$
It is well known that, for any continuous function~$h:\R_+\times I\to \R$ such that 
\begin{equation*} \label{gamma-norm}
|\!|\!|h|\!|\!|_\gamma:=\sup_{t\ge0,\, x\in I}|h(t,x)|+\sup_{t_1,t_2\ge0\atop x_1,x_2\in I}\frac{|h(t_1,x_2)-h(t_2,x_2)|}{|t_1-t_2|^{\gamma/2}+|x_1-x_2|^\gamma}<\infty,	
\end{equation*}
and any initial state $u_0\in C^\gamma(I)$ satisfying an appropriate compatibility condition (see~\eqref{compatibility}), problem~\eqref{burgers}--\eqref{IC-1D} has a unique solution $u:\R_+\times I\to\R$, continuously differentiable in~$t$ and twice continuously differentiable in~$x$. Moreover, $u$~belongs to the space $\XX^\gamma$ of continuous functions $v:\R_+\times I\to\R$ whose derivatives $\p_t v,\p_xv,\p_x^2v$ exist and have finite~$|\!|\!|\cdot|\!|\!|_\gamma$ norms.

We now fix an arbitrary non-negative smooth function $\varphi\in C_0^\infty(I)$ that does not vanish identically and has compact support, and consider the controlled Burgers equation
\begin{equation*} \label{controlled-burgers}
	\p_tu+\p_xA(u)-\nu\p_x^2u=h(t,x)+\xi(t)\varphi(x), \quad x\in I,
\end{equation*}
where $\xi$ is a control function. Given an arbitrary solution $\hat u\in\XX^\gamma$ of~\eqref{burgers}, \eqref{dirichlet} and an initial condition $u_0\in C^\gamma(I)$, we wish to find a piecewise-constant function $\xi:\R_+\to\R$ such that the solution $u\in\XX^\gamma$ of~\eqref{burgers}--\eqref{IC-1D} satisfies the inequality
\begin{equation*} \label{expo-stabilisation}
	\|u(t)-\hat u(t)\|_{C^2}\le Ce^{-\alpha t},
\end{equation*}
where~$C$ and~$\alpha$ are some positive numbers, and the $C^2$ norm is defined as the maximum of the sum of the space  derivatives up to the second order. If this property holds for any reference solution~$\hat u$, then we shall say that Eq.~\eqref{burgers} is {\it exponentially stabilisable\/} in~$C^2$ by a one-dimensional external force (proportional to~$\varphi$). The following theorem is a particular case of the main result of this paper formulated in Theorem~\ref{t-control}. 

\begin{mtheorem}
	In addition to the above hypotheses, suppose that the function $A\in C^2(\R)$ satisfies the inequality
	\begin{equation*}
		A'(u)u\ge c\,u^2-C, \quad u\in \R,
	\end{equation*}
	where $C$ and $c$ are some positive numbers. Then Burgers' equation~\eqref{burgers} is exponentially stabilisable in the space~$C^2$. 
\end{mtheorem}

It should be noted that the problem of control and stabilisation of Burgers' equation and more general viscous conservation laws attracted a lot of attention in the last thirty years. In particular, Fursikov and Imanuvilov~\cite{FI-1995} studied the 1D Burgers equation~\eqref{burgers} with quadratic flux~$A(u)=\frac12u^2$ and proved the local exact controllability and the absence of global approximate controllability. The latter property was refined by Diaz~\cite{diaz-1996}, and further non-controllability results were obtained by Guererro--Imanuvilov~\cite{GI-2007}. Glass and Guererro~\cite{GG-2007}, L\'eautaud~\cite{leautaud-2012}, Coron~\cite{coron-2007}, and Fern\'andez-Cara--Guererro~\cite{FG-2007} proved global exact boundary controllability to constant states for Burgers' equation and established some estimates for the time and cost of control. Kr\"oner--Rodrigues~\cite{KR-2015} studied the stabilisation to a non-stationary solution with a control localised in the Fourier space, and the second author~\cite{shirikyan-jep2017} proved the stabilisation of non-stationary trajectories by a control with support in an interval. In the two-dimensional case, Thevenet--Buchot--Raymond~\cite{TBR-2010} established the exponential stabilisation of stationary solutions by a feedback control. In conclusion, let us mention that various results on controllability and stabilisation were established for other equations of fluid mechanics, such as the Euler, Navier--Stokes, and Boussinesq systems (see the books~\cite{FI1996,fursikov2000,coron2007} and the references therein). However, we do not mention them here, since they are not directly related to the subject of this paper, whose methods are based on the theory of positivity-preserving operators and are not likely to be applicable in the case of the above-mentioned PDEs.

\smallskip
In the context of damped conservation laws, many questions remain open. We formulate only two of them, closely related to the main result of  this paper. 

\begin{op}
Theorem~\ref{t-control} shows that, given a trajectory~$\hat u(t,x)$ of a viscous conservation law and an arbitrary initial condition, one can find a one-dimensional localised control bringing the corresponding solution~$u(t,x)$ arbitrarily close to~$\hat u$ at some finite time. In the case of Burgers' equation, the Fursikov--Imanuvilov local controllability result~\cite{FI1996} implies that the two trajectories can be made to be equal with the help of a localised control. Imanuvilov and Yamamoto~\cite{IY-2003} established a similar local exact controllability result in the multidimensional case for a nonlinear parabolic equation, assuming that the nonlinearity has at most linear growth at infinity. It is a challenging open problem to prove the property of exact controllability (in the multidimensional case) with no growth conditions on the flux~$A$. 
\end{op}

\begin{op}
The local-in-time $L^1$ norm of the one-dimensional stabilising control obtained in Theorem~\ref{t-control} goes to zero exponentially fast as $t\to\infty$. On the other hand, it follows from~\eqref{control-estimate} that the size of control is of order~$1$ on some intervals. From the point of view of applications, it would be important to prove that the control can be chosen in a feedback form, so that its size decays to zero as $t\to\infty$. This type of (local) results are well known in the case of stationary reference trajectories; see~\cite{fursikov-2004,raymond-2007}. 
\end{op}

The paper is organised as follows. In Section~\ref{s-IBVP}, we discuss the well-posedness of the damped-driven conservation law and establish some estimates for solutions. In Section~\ref{s-controllability}, we formulate and prove the main result of this paper. The appendix gathers some known results used in the main text. 

\subsubsection*{Acknowledgements}
The research of the first author was funded by Deutsche Forschungsgemeinschaft (DFG) through the grant CRC 1114: {\it Scaling Cascades in Complex Systems\/}, Project Number 235221301. The research of the second author was carried out within the MME-DII Center of Excellence (ANR-11-LABX-0023-01) and supported by the \textit{Agence Nationale de la Recherche\/} through the grant NONSTOPS (ANR-17-CE40-0006) and by the \textit{CY Initiative of Excellence\/} through the grant {\it Investissements d'Avenir\/} ANR-16-IDEX-0008. 

\subsubsection*{Notation}
We denote by~$D$ a connected bounded domain  in~$\R^d$ with a~$C^3$ boundary~$\p D$ and  by~$J_T$ the time interval $[0,T]$ and write $Q_T=J_T\times D$. Given a number~$\gamma\in(0,1)$, we denote by $C^\gamma(D)$ the space of continuous functions $g:D\to\R$ such that 
$$
\|g\|_{C^\gamma}:=\sup_{x\in D}|g(x)|+\sup_{x_1,x_2\in D}\frac{|g(x_1)-g(x_2)|}{|x_1-x_2|^\gamma}<\infty. 
$$
Similarly, if $k\ge1$ is an integer, then~$C^{k+\gamma}(D)$ stands for the space of functions $g\in C^\gamma(D)$ whose derivatives up to the order~$k$ belong to~$C^\gamma(D)$. The corresponding norm is defined by the formula
$$
\|g\|_{C^{k+\gamma}}:=\sum_{|\alpha|\le k}\|\p_x^\alpha g\|_{C^\gamma}. 
$$
For a cylinder $Q=J\times D$ (where $J\subset\R$ is an interval), we write~$C^{1,2}(\overline Q)$ for the space of continuous functions $h:\overline Q\to\R$ that are continuously differentiable in~$t$ and twice continuously differentiable in~$x$. For $\gamma\in(0,1)$, we denote by~$C^{\frac{\gamma}{2},\gamma}(Q)$ the space of continuous functions~$h$ such that 
$$
|\!|\!|h|\!|\!|_\gamma:=\sup_{(t,x)\in Q}|h(t,x)|+\sup_{t_1,t_2\in J\atop x_1,x_2\in D}\frac{|h(t_1,x_2)-h(t_2,x_2)|}{|t_1-t_2|^{\gamma/2}+|x_1-x_2|^\gamma}<\infty. 
$$
We write~$\XX^\gamma(Q)$ for the space of functions in~$C^{1,2}(\overline Q)$ whose first-order derivative in~$t$ and second-order derivatives in~$x$ belong to~$C^{\frac{\gamma}{2},\gamma}(Q)$. In the case when~$J$ is replaced by~$\R_+$, the spaces mentioned above will be denoted by~$C^{1,2}$, $C^{\frac{\gamma}{2},\gamma}$, and~$\XX^\gamma$, respectively. Finally, we denote by $C_b$ the space of bounded continuous functions on the cylinder~$\R_+\times\overline D$. 

\section{Preliminaries on multidimensional viscous conservation laws}
\label{s-IBVP}

\subsection{Existence and uniqueness of solution}
\label{ss-existence}
Let us consider the following initial-boundary value problem in a bounded domain $D\subset\R^d$ with a boundary~$\p D\in C^3$:
\begin{align}
	\p_t u+\diver_x A(u)-\nu\Delta u
	&=h(t,x),
	\label{PDE}\\
	u\bigr|_{\p D}&=0,\label{BC}\\
	u(0,x)&=u_0(x).\label{IC}
\end{align}
Here $\nu>0$ is a parameter, $h$ and~$u_0$ are continuous functions in~$\R_+\times \overline D$ and~$\overline D$, respectively, and $A:\R\to\R^d$ is a twice continuously differentiable vector function such that $A(0)=0$.  We are interested in the existence and uniqueness of solutions for problem~\eqref{PDE}--\eqref{IC} in  H\"older spaces. 
\begin{theorem} \label{t-PDE-IVP}
	Let $\gamma\in(0,1)$, $T>0$, and $\nu>0$ be any numbers. Then, for any data $h\in C^{\frac{\gamma}{2},\gamma}(Q_T)$ and $u_0\in C^{2+\gamma}(D)$ satisfying the compatibility conditions
	\begin{equation} \label{compatibility}
	u_0=0, \quad 
	\diver_xA(u_0)-\nu \Delta u_0-h(0)=0 \quad\mbox{on $\p D$},
	\end{equation}
problem \eqref{PDE}--\eqref{IC} possesses a unique solution  $u\in \XX^\gamma(Q_T)$. Moreover, there is a number $M>0$ depending continuously on~$T$, $|\!|\!|h|\!|\!|_\gamma$, and $\|u_0\|_{C^{2+\gamma}}$ such that 
\begin{equation} \label{bound}
	\|u\|_{\XX^\gamma(Q_T)}\le M. 
\end{equation}
\end{theorem}

\begin{proof}
We claim that Theorem 6.1 in~\cite[Chapter~V]{LSU1968} on quasilinear parabolic PDEs of divergence form is applicable in our situation. Indeed, in view of that result, setting 
$$
\aaaa(t,x,u,p)=\sum_{j=1}^dA_j'(u)p_j-h(t,x),
$$
we only need to check that $|\aaaa(t,x,u,0)|\le C$ for all $(t,x,u)\in Q_T\times\R$. This inequality is obvious, and the existence and uniqueness of solution in the space~$\XX^\gamma(Q_T)$ follows. Inequality~\eqref{bound} is established in the proof of Theorem 6.1 in~\cite[Chapter~V]{LSU1968}. 
\end{proof}

\begin{remark} \label{r-existence}
	The above result on the existence and uniqueness of solution remains valid if the right-hand side~$h$ of~\eqref{PDE} is a function with the following property: there is partition $0=t_0<t_1<\cdots<t_n=T$ of the interval~$J_T$ such that, for any $k\in[\![1,n]\!]$, the restriction of~$h$ to $Q_k:=(t_{k-1},t_k)\times D$ belongs to~$C^{\frac{\gamma}{2},\gamma}(Q_k)$. Indeed, it suffices to apply Theorem~\ref{t-PDE-IVP} to each of the cylinders~$Q_k$ and to consider the concatenation of the resulting solutions. The trajectory obtained by that procedure will not be an element of~$\XX^\gamma(Q_T)$, however, it will be a continuous curve on~$J_T$ with range in~$C^{2+\gamma}(D)$. 
\end{remark}

In what follows, we shall need various bounds on solutions for~\eqref{PDE} expressed in terms of some norms of the right-hand side~$h$. To this end, we first  recall a comparison principle for nonlinear parabolic PDEs and then apply it to derive required a priori estimates for solutions.

\subsection{Comparison principle}
\label{ss-comparison}
Let us consider the nonlinear differential operator
\begin{equation*} \label{DD-viscous-law}
	F(u):=\p_tu+\sum_{j=1}^d\p_j A_j(t,x,u)-\nu\Delta u,
\end{equation*}
where $A=(A_1,\dots,A_d)$ is a $C^1$ vector function, and the partial derivative~$\p_j$ under the sum applies to both the $x$-variable and the $x$-argument of~$u$. We shall say that a pair of functions $(u^+,u^-)$ belonging to the space $C^{1,2}(\overline Q_T)$ is {\it admissible\/} for~$F(u)$ if 
\begin{equation} \label{admissible-pair}
F(u^+)\ge F(u^-)\quad\mbox{in~$Q_T$}. 
\end{equation}
Recall that the {\it parabolic boundary\/} of~$Q_T$ is defined by 
$$
\p_pQ_T=\{(t,x)\in\R^{d+1}:\mbox{$t=0$, $x\in D$ or $t\in J_T$, $x\in \p D$}\}.
$$
The following result is established in~\cite{AL-1983} for weak solutions under more restrictive hypotheses. Even though its proof is based on well known ideas, we briefly sketch it for the reader's convenience. 

\begin{theorem} \label{t-comparison-principle}
	Let $A\in C^1(\R^{d+2},\R^d)$ and let~$(u^+,u^-)$ be an admissible pair for~$F(u)$ such that $u^+\ge u^-$ on~$\p_pQ_T$. Then $u^+\ge u^-$ in~$Q_T$. 
\end{theorem}

\begin{proof}
	The function $u=u^+-u^-$ is non-negative on the parabolic boundary boundary~$\p_pQ_T$ and satisfies the differential inequality
	\begin{equation} \label{differential-inequality}
		\p_tu-\nu\Delta u+\diver(b(t,x)u)\ge0,
	\end{equation}
	where we used~\eqref{admissible-pair} and set
	$$
	b(t,x)=\int_0^1 (\p_u A)\bigl(t,x,u^-(t,x)+\theta u(t,x)\bigr)\,\dd\theta. 
	$$
	We need to prove that $u(s,x)\ge0$ for any $(s,x)\in Q_T$. To this end, it suffices to show that, for any non-negative function $\varphi\in C_0^\infty(D)$,
	\begin{equation} \label{integral-inequality}
		\int_D u(s,x)\varphi(x)\,\dd x\ge0. 
	\end{equation}
	
	Together with~\eqref{differential-inequality}, let us consider the dual problem
	\begin{align}
		\p_t\psi+\nu\Delta \psi+\langle b(t,x),\nabla\rangle \psi&=0,\label{dual-PDE}\\
		\psi(s,x)=\varphi(x), \quad \psi\bigr|_{\p D}&=0\label{IB-condition}
	\end{align}
	in the domain $(0,s)\times D$. Since $\varphi\ge0$, by the maximum principle (see Section~3.2 in~\cite{landis1998}), we have $\psi\ge0$, so that $\p_n\psi\le0$ on~$\p D$, where~$n$ stands for the outward unit normal vector to~$\p D$. It follows from~\eqref{differential-inequality} and~\eqref{dual-PDE} that
	$$
	\frac{\dd}{\dd t}(u,\psi)=(\p_t u,\psi)+(u,\p_t \psi)=
	-\int_{\p D}u\p_n\psi\,\dd\sigma \ge0. 
	$$
	Integrating in $t\in[0,s]$, using~\eqref{IB-condition}, and recalling that $\varphi\ge0$ and $u(0)\ge0$, we obtain the required inequality~\eqref{integral-inequality}. 
\end{proof}

\subsection{A priori estimates}
\label{ss-apriori}
Let us consider the PDE~\eqref{PDE}, in which~$h(t,x)$ is a given bounded continuous function, and the vector function $A\in C^2(\R,\R^d)$ (which now depends only on~$u$) satisfies the following sign condition.

\begin{itemize}
	\item [\hypertarget{(S)}{\bf(S)}]
	\sl There are positive numbers $c$ and~$C$ such that, for all $u\in\R$, 
	\begin{equation} \label{sign-condition}
		\biggl(\sum_{j=1}^d A_j'(u)\biggr)\sgn(u)\ge c|u|-C. 
	\end{equation}
\end{itemize}
The following result is inspired by Lemma~9 in~\cite[Section~2.1]{coron-2007}, which deals with the 1D Burgers equation. 

\begin{proposition} \label{p-comparison-DD-conservation}
	 Let $A\in C^2(\R,\R^d)$ be a vector function satisfying~\hyperlink{(S)}{\rm(S)} and let $h \in C_b$. Then, for any solution $u\in C^{1,2}$ of~\eqref{PDE}, \eqref{BC} and any $T>0$, we have 
	 \begin{equation}\label{L-infty-bound}
	 	\|u(T)\|_{L^\infty}\le M,
	 \end{equation}
	 where $M>0$ is a number depending only on~$T$, $C$, $c$, and~$\|h\|_{L^\infty}$.
\end{proposition}

\begin{proof}
	Given $\e>0$, we define 
	$$
	u_\e^+(t,x)=\frac{B_\e\bigl(B_\e+\sum_{j=1}^dx_j-d\alpha^-\bigr)+L\e}{t+\e},
	$$
	where $L=\|u(0)\|_{L^\infty}$, $\alpha^-=\min\{x_j:x\in\overline D,1\le j\le d\}$, and $B_\e>0$ is a continuous increasing function of~$\e\in[0,1]$. We claim that, for an appropriate choice of~$B_\e$, the pair $(u_\e^+,u)$ is admissible for~$F$ and satisfies the inequality $u_\e^+\ge u$ on~$\p_pQ_T$. Indeed, the function~$u_\e^+$ is positive and satisfies the inequality $u_\e^+(0)\ge \|u_0\|_{L^\infty}$. Furthermore, a simple calculation based on~\eqref{sign-condition}  shows that
	$$
	F(u_\e^+)(t,x)\ge\frac{1}{t+\e}\bigl((B_\e c-1)u_\e^+(t,x)-B_\e C\bigr).
	$$
	We wish to prove that the right-hand side of this inequality is greater than~$\|h\|_{L^\infty}$ for all $(t,x)\in Q_T$. It is not difficult to check that this is certainly the case if
	$$
	B_\e=\max\bigl\{2c^{-1},2C(T+\e),C^{-1}(T+\e)\|h\|_{L^\infty}\bigr\}.
	$$
	Thus, the conclusion of Theorem~\ref{t-comparison-principle} is applicable to the pair~$(u_\e^+,u)$ for any $\e>0$, whence it follows that 
	\begin{equation} \label{UB-comparison}
	\max_{x\in \overline D}u(T,x)\le\lim_{\e\to0^+}\max_{x\in \overline D}u_\e^+(T,x)
	\le M:=T^{-1}B_0(B_0+d(\alpha^+-\alpha^-)),
	\end{equation}
	where $\alpha^+=\max\{x_j:x\in\overline D, 1\le j\le d\}$. 
	
	To derive a lower bound, it suffices to apply a similar argument to the pair~$(u,u_\e^-)$, where
	$$
	u_\e^-(t,x)=-\frac{B_\e\bigl(B_\e-\sum_{j=1}^dx_j+d\alpha^+\bigr)+L\e}{t+\e}.
	$$
	This results in the inequality 
	$$
	\min_{x\in \overline D}u(T,x)\ge\lim_{\e\to0^+}\min_{x\in \overline D}u_\e^-(T,x)
	\ge -M,
	$$
	where $M$ is the same constant as in~\eqref{UB-comparison}. The proof of the proposition is complete.
\end{proof}

In what follows, we always assume that Condition~\hyperlink{(S)}{\rm(S)} is fulfilled and do not follow the dependence of various quantities on the numbers~$c$ and~$C$ entering~\eqref{sign-condition}. 

\begin{corollary} \label{c-uniform-bounds}
	Let $A\in C^2(\R,\R^d)$ be a vector function satisfying Condition~\hyperlink{(S)}{\rm(S)}, and let $h\in C^{\frac{\gamma}{2	},\gamma}$ and $u_0\in C^{2+\gamma}(D)$ be such that the compatibility conditions~\eqref{compatibility} hold. Then there is~$K>0$ depending continuously on~$|\!|\!|h|\!|\!|_\gamma$ and~$\|u_0\|_{C^{2+\gamma}}$ such that the solution~$u\in\XX^\gamma$ of problem~\eqref{PDE}--\eqref{IC} satisfies the inequality
	\begin{equation} \label{global-bound}
		\|u\|_{\XX^\gamma}\le K. 
	\end{equation}
	Moreover, for any $\tau>0$ there is $K_\tau>0$ depending only on~$|\!|\!|h|\!|\!|_\gamma$ such that 
	\begin{equation} \label{universal-bound}
		\|u(\tau+\cdot)\|_{\XX^\gamma}\le K_\tau.
	\end{equation}	
\end{corollary}

\begin{proof}
	This result is a consequence of the uniform bound~\eqref{L-infty-bound}, the a priori estimate~\eqref{bound}, and the parabolic regularisation.  Namely, suppose for any $s>0$ we found a function $C_s:\R_+^2\to\R_+$ increasing in both arguments such that any solution $u\in \XX^\gamma$ satisfies the inequality 
	\begin{equation} \label{regularisation}
		\|u(s)\|_{C^{2+\gamma}}\le C_s\bigl(|\!|\!|h|\!|\!|_\gamma,\|u_0\|_{L^\infty}\bigr). 
	\end{equation}
	Using this inequality with $t=\tau/2$ and inequality~\eqref{L-infty-bound}  with $T\ge\tau/2$, we see that $\|u(t)\|_{C^{2+\gamma}}\le C_{\tau/2}(|\!|\!|h|\!|\!|_\gamma,M)$ for any $t\ge\tau$. The a priori bound~\eqref{bound} now implies that~\eqref{universal-bound} holds with some $K_\tau>0$. Inequality~\eqref{global-bound} follows immediately from~\eqref{universal-bound} and~\eqref{bound}. 
	
	To complete the proof, it remains to establish inequality~\eqref{regularisation}. Its proof is based on parabolic regularisation and a standard bootstrap argument, and we only outline it. We fix any $s>0$ and use Theorem~10.1 in~\cite[Chapter~III]{LSU1968} to find $\sigma\in(0,\gamma)$ such that $u\in C^{\frac{\sigma}{2},\sigma}(\Omega_s)$ and 
	\begin{equation*} \label{bound-intermediate}
		\|u\|_{C^{\frac{\sigma}{2},\sigma}(\Omega_s)}\le M_1\bigl(|\!|\!|h|\!|\!|_\gamma,\|u_0\|_{L^\infty}\bigr), 
	\end{equation*}
	where $\Omega_s=(s/3,s)\times D$, and~$M_i$ stand for some increasing functions of their arguments. Rewriting~\eqref{PDE} as the linear parabolic equation
	\begin{equation} \label{linear-parabolic}
		\p_tu-\nu\Delta u+\sum_{j=1}^da_j(t,x)\p_ju=h(t,x),
	\end{equation}
	where $a_j=A_j'\circ u\in C^{\frac{\sigma}{2},\sigma}(\Omega_s)$, we can use inequalities~\eqref{green-bound} and~\eqref{green-holder} for Green's function to conclude that $u\in\XX^\sigma(\Omega_s')$, where $\Omega_s'=(s/2,s)\times D$, and the norm $\|u\|_{\XX^\sigma(\Omega_s')}$ is bounded by $M_2(|\!|\!|h|\!|\!|_\gamma,\|u_0\|_{L^\infty})$. It follows that the coefficients~$a_j$ in~\eqref{linear-parabolic} belong to~$C^{\frac{\gamma}{2},\gamma}(\Omega_s')$, and their norms are bounded by $M_3(|\!|\!|h|\!|\!|_\gamma,\|u_0\|_{L^\infty})$.  Applying the above argument with~$\sigma$ replaced by~$\gamma$, we arrive at the required bound~\eqref{regularisation}. 
\end{proof}

\section{Stabilisation by a one-dimensional control}
\label{s-controllability}

\subsection{Formulation of the main result}
\label{s-control-result}
Let $D\subset\R^d$ be a bounded domain with a  boundary $\p D\in C^3$. We consider the controlled PDE
\begin{equation} \label{controlled-PDE}
	\p_t u+\diver_x A(u)-\nu\Delta u=h(t,x)+\eta(t,x),
\end{equation}
supplemented with the initial and boundary conditions~\eqref{BC}, \eqref{IC}. Here $\nu>0$ is a fixed parameter, $A\in C^2(\R,\R^d)$ is a function satisfying~\eqref{sign-condition}, $h:\R_+\times D\to\R$ is a given function in $C^{\frac{\gamma}{2},\gamma}$, and~$\eta$ is a control supported in a fixed domain~$\Pi$ such that $\overline\Pi\subset D$.

\begin{theorem} \label{t-control}
	Let the above hypotheses be satisfied, let~$\hat u$ be a solution of~\eqref{PDE}, \eqref{BC} that belongs to the space $\XX^\gamma$, let $\varphi\in C_0^\infty(\Pi)$ be a non-negative function that is not identically zero, and let $\sigma\in(0,\gamma)$. Then there are positive numbers~$\alpha$, $\beta<1$, and~$C$, not depending on~$\hat u$, such that for any $u_0\in C^{2+\gamma}(D)$ satisfying the compatibility condition~\eqref{compatibility} one can find a control $\eta(t,x)=\xi(t)\varphi(x)$, with a piece-wise constant~$\xi$, for which the solution~$u(t,x)$ of problem~\eqref{controlled-PDE}, \eqref{BC}, \eqref{IC} satisfies the inequality
	\begin{equation} \label{expo-stab}
		\bigl\|u(t)-\hat u(t)\bigr\|_{C^{2+\sigma}}\le Ce^{-\alpha t}\min\bigl\{\|u_0-\hat u(0)\|_{L^1}^\beta, 1\bigr\}, \quad t\ge1.
	\end{equation}
	Moreover, there are positive numbers $K$, $\varkappa$ and sequences $\xi_k\in\R$, $\tau_k\in(0,1/4)$ such that 
	\begin{gather}
		|\xi_k|=K, \quad \tau_k=\varkappa\,\|u(k-1)-\hat u(k-1)\|_{L^1}\quad\mbox{for any $k\ge1$}, 
		\label{control-estimate}\\[2pt]
				\xi(t)=\left\{
		\begin{array}{cl}
			0  & \mbox{for $t\in\bigl[k-1,k-\tfrac12\bigr)\cup \bigl[k-\tfrac12+\tau_k,k\bigr)$},\\[3pt]
			\xi_k & \mbox{for $t\in \bigl[k-\tfrac12,k-\tfrac12+\tau_k\bigr)$}. 
		\end{array}
		\right.\label{control-structure}
	\end{gather}
\end{theorem}

A straightforward consequence of the above theorem is the following approximate controllability result. 

\begin{corollary} \label{c-AC}
	Under the hypotheses of Theorem~\ref{t-control}, for any $\e>0$, $\gamma\in(0,1)$,  and $\sigma\in(0,\gamma)$ there is $T>0$ such that, given initial states $u_0,\hat u_0\in C^{2+\gamma}(D)$ satisfying the compatibility condition~\eqref{compatibility}, one can find a piecewise constant function~$\xi:J_T\to\R$ such that 
	\begin{equation*} \label{AP}
		\bigl\|u(t)-\hat u(t)\bigr\|_{C^{2+\sigma}}\le\e, \quad 0\le t\le T,
	\end{equation*}
	where $\hat u\in\XX^\gamma(Q_T)$ denotes the solution of problem~\eqref{PDE}--\eqref{IC} with $u_0=\hat u_0$, and $u:J_T\to C^{2+\gamma}(D)$ stands for the solution of the control system~\eqref{controlled-PDE}, \eqref{BC}, \eqref{IC} with $\eta(t,x)=\xi(t)\varphi(x)$. 
\end{corollary}

Let us sketch the proof of Theorem~\ref{t-control}, postponing the details to the next two subsections. Note that, by the Gagliardo--Nirenberg--Sobolev interpolation inequality (see~\cite{BM-2019}), for any $0<\sigma<\gamma<1$ and $f\in C^{2+\gamma}(D)$ we have 
\begin{equation} \label{GNS-inequality}
	\|f\|_{C^{2+\sigma}}\le C\,\|f\|_{L^1}^\theta \|f\|_{C^{2+\gamma}}^{1-\theta}, 
\end{equation}
where $\theta=\frac{\gamma-\sigma}{d+\gamma+2}$. Therefore, the exponential stabilisation~\eqref{expo-stab} will be established if we find a control $\eta(t)$ such that the $L^1$-norm of the difference $u(t)-\hat u(t)$ goes to zero exponentially as $t\to\infty$, while the $C^{2+\gamma}$-norm remains bounded. We confine ourselves to the main part of the proof dealing with the exponential decay of the $L^1$-norm. 

Let us consider the difference $w=v-\hat u$, where~$v$ is another solution of Eq.~\eqref{PDE}. Then~$w$ must be a solution of the  linear problem
\begin{equation} \label{linear-uncontrolled}
	\p_t w-\nu\Delta w+\diver \bigl(a(t,x)w\bigr)=0, \quad w\bigr|_{\p D}=0, 
\end{equation}
where we set
\begin{equation} \label{function-a}
a(t,x)=\int_0^1A'\bigl(\hat u(t,x)+s w(t,x)\bigr)\,\dd s. 	
\end{equation}
A key property that will enable us to prove the theorem is that either the $L^1$-norm of~$w$ squeezes by the time $t=\frac12$, or $|w(\frac12,x)|$ is minorised on~$\Pi$ by a positive constant. Namely, we have the following result. 

\begin{proposition} \label{p-dichotomy}
	Let $T>0$ and $\rho>0$ be some numbers and let $a\in C(\overline Q_T,\R^d)$ be a function continuously differentiable in~$x$ such that
	\begin{equation} \label{a-bound}
		\sup_{(t,x)\in Q_T}\bigl(|a(t,x)|+|\nabla_x a(t,x)|\bigr)\le\rho.
	\end{equation}
	Then there are numbers $q<1$ and~$\delta>0$ depending only on~$T$, $\rho$, $\nu$ and~$\Pi$ such that any solution $w\in C^{1,2}(\overline Q_T)$ of~\eqref{linear-uncontrolled} satisfies one of the following inequalities:
	\begin{align} 
	\|w(T)\|_{L^1}&\le q\|w(0)\|_{L^1},\label{L1-squeezing}\\
	\inf_{x\in\Pi}|w(T,x)|&\ge \delta \|w(0)\|_{L^1} . 
	\label{lower-bound}
	\end{align}
\end{proposition}

In view of Corollary~\ref{c-uniform-bounds}, inequality~\eqref{a-bound} holds for the vector function $a(t,x)$ defined by~\eqref{function-a}. Applying Proposition~\ref{p-dichotomy} with $T=\frac12$ to the function $w=v-\hat u$, we distinguish between two cases. If~\eqref{L1-squeezing} holds, then we choose $\xi\equiv 0$, so that the solution~$u$ of the controlled problem~\eqref{controlled-PDE}, \eqref{IC} coincides with~$v$ and, hence, in view of the contraction of the $L^1$-norm, satisfies the inequality
\begin{equation*} \label{squeezing-u-hatu}
	\|u(1)-\hat u(1)\|_{L^1}\le \|u(\tfrac12)-\hat u(\tfrac12)\|_{L^1}
	\le q\,\|u_0-\hat u(0)\|_{L^1}. 
\end{equation*}
In the opposite case, $w$ satisfies~\eqref{lower-bound} with $T=\frac12$. Let us fix a parameter $\tau\in(0,\tfrac12)$, set $T_\tau=\tfrac12+\tau$, and define the function $\eta(t,x)=\xi(t)\varphi(x)$, where~$\xi$ is zero on the set $[0,\tfrac12)\cup[T_\tau,1)$ and is equal to a contant~$\bar\xi$ on $[\tfrac12,T_\tau)$. The difference $z=u-\hat u$ is a solution of the problem 
\begin{align} \label{z-problem}
		\p_t z-\nu\Delta z+\diver \bigl(A(\hat u+z)-A(\hat u)\bigr)=\bar\xi\,\varphi(x), \quad z\bigr|_{\p D}=0, \quad z(\tfrac12)=\bar z
\end{align}
on the interval $[\frac12,T_\tau]$, where $\bar z=w(\tfrac12)$. Note that, in view of the $L^1$-contraction for~\eqref{linear-uncontrolled} and inequality~\eqref{lower-bound} with $T=\frac12$, we have 
\begin{equation} \label{L1-z}
	\|\bar z\|_{L^1}\le d, \quad \inf_{x\in\Pi}|\bar z(x)|\ge \delta d,
\end{equation}
where $d=\|u_0-\hat u(0)\|_{L^1}$. A key observation is that the solution of~\eqref{z-problem} is close to that of the ODE
\begin{equation} \label{Z-ODE}
	\p_t Z=\bar\xi\,\varphi(x),  \quad Z(\tfrac12)=\bar z.
\end{equation}
Namely, we have the inequality 
\begin{equation} \label{difference-z-Z}
	\|z(T_\tau)-Z(T_\tau)\|_{L^1}\le C_1\tau,
\end{equation}
where $C_1>0$ does not depend on~$u_0$, $\hat u(0)$, and~$\tau$; see Lemma~\ref{l-comparison} below. In view of the second inequality in~\eqref{L1-z} and the continuity of~$\bar z$, we can assume, for instance, that $\bar z(x)\ge \delta d$ for $x\in\Pi$. In this case, we set $\bar \xi=-(\tau M)^{-1}\theta\delta d$, where $M=\|\varphi\|_{L^\infty}$ and $\theta\in(0,1)$ is a small parameter to be fixed later. It follows from \eqref{L1-z}--\eqref{difference-z-Z} that 
\begin{align}
	\|u(T_\tau)-\hat u(T_\tau)\|_{L^1}=\|z(T_\tau)\|_{L^1}
	&\le \|z(T_\tau)-Z(T_\tau)\|_{L^1}+\|Z(T_\tau)\|_{L^1}\notag\\
	&\le C_1\tau + \|\bar z-M^{-1}\theta\delta d\,\varphi\|_{L^1}\notag\\
	&\le C_1\tau+d\bigl(1-M^{-1}\theta\delta \|\varphi\|_{L^1}\bigr). 
\end{align}
Choosing $\tau=\e d$ with a small $\e>0$ and using again the $L^1$-contraction for~\eqref{linear-uncontrolled}, we obtain
\begin{equation} \label{squeezing-final}
	\|u(1)-\hat u(1)\|_{L^1}\le \|u(T_\tau)-\hat u(T_\tau)\|_{L^1}
	\le q_1\|u_0-\hat u(0)\|_{L^1},
\end{equation} 
where $q_1=1+C_1\e-M^{-1}\theta\delta \|\varphi\|_{L^1}$. Hence, if $\e=(2C_1M)^{-1}\theta\delta\|\varphi\|_{L^1}$, then $q_1=1-\frac12M^{-1}\theta\delta \|\varphi\|_{L^1}<1$. Iteration of this procedure will give the required exponential decay of the $L^1$-norm of the difference $u(t)-\hat u(t)$. 

\subsection{Proof of Proposition~\ref{p-dichotomy}}
\label{s-auxiliary-result}
The proof builds on the argument used in~\cite[Section~3.2]{shirikyan-jep2017} for the 1D Burgers equation. Since the result we establish here is stronger, we provide all the details.  

\smallskip
{\it Step~1}. We first show that for any $q\in(0,1)$ there is $\e=\e(\Pi,T,q,\rho)>0$ such that if $w\in L^\infty(Q_T)$ is a non-negative solution of problem~\eqref{linear-uncontrolled}, then either~\eqref{L1-squeezing} holds, or
\begin{equation} \label{inf-bound}
\inf_{x\in \Pi}w(T,x)\ge \e\,\|w(0)\|_{L^1}. 
\end{equation}
Indeed, there is no loss of generality in assuming that $\|w(0)\|_{L^1}=1$. Denoting by $G(t,x,y)$ Green's function for~\eqref{linear-uncontrolled} and writing~$w$ in the form
\begin{equation} \label{w-green}
w(t,x)=\int_DG(t,x,y)w(0,y)\,\dd y,	
\end{equation}
we note that, in view of inequality~\eqref{green-bound} with $j=|\alpha|=0$, 
\begin{equation} \label{upper-bound-w}
\sup_{(t,x)\in[\frac{T}{2},T]\times D}w(t,x)\le M, 
\end{equation}
where we set 
$$
M=\sup_{t\in[\frac{T}{2},T]}\sup_{x,y\in D}G(t,x,y). 
$$
Let $K\subset D$ be a connected compact set  such that  $\Pi\subset K$ and 
$\mathscr{L}(D\setminus K)\le \frac{q}{2M}$, where $\mathscr{L}$ stands for the Lebesgue measure on~$\R^d$. By Harnack's inequality (see Proposition~\ref{p-Harnack}), we can find $C=C(K,T)>0$ such that
\begin{equation} \label{2.33}
\sup_{x\in K}w\bigl(\tfrac{T}{2},x\bigr)\le C\inf_{x\in K}w(T,x). 
\end{equation}
We now set $\e=q(2C\mathscr{L}(K))^{-1}$ and suppose that
\begin{equation*} 
\inf_{x\in \Pi}w(T,x)\le \e;
\end{equation*}
in the opposite case~\eqref{inf-bound} holds. Since $\Pi\subset K$, we have
\begin{equation} \label{2.34} 
\inf_{x\in K}w(T,x)\le \e. 
\end{equation}
Combining~\eqref{upper-bound-w}--\eqref{2.34} and the contraction of the $L^1$-norm for solutions of~\eqref{linear-uncontrolled}, we derive
\begin{align*}
\|w(T)\|_{L^1}
&\le \|w(T/2)\|_{L^1}\le \int_{D\setminus K}w(T/2)\,\dd x+\int_K w(T/2)\,\dd x\\
&\le  M \mathscr{L}(D\setminus K)+C\e\mathscr{L}(K)\le q. 
\end{align*}
This coincides with~\eqref{L1-squeezing} in the case under study. 

\medskip
{\it Step~2\/}. We now consider the case of an arbitrary solution, assuming again that $\|w(0)\|_{L^1}=1$. Let us denote by~$w_0^+$ and~$w_0^-$ the positive and negative parts of~$w(0)$ and write~$w^\pm$ for the solution of~\eqref{linear-uncontrolled} issued from~$w_0^\pm$. We set $r=\|w_0^+\|_{L^1}$ and assume without loss of generality that $r\ge\frac12\|w_0\|_{L^1}=\frac12$. Suppose first that $\|w^+(T)\|_{L^1}\le r/2$. Then, in view of the contraction of the $L^1$-norm of solutions for~\eqref{linear-uncontrolled}, we have 
\begin{align*}
\|w(T)\|_{L^1}
\le \|w^+(T)\|_{L^1}+\|w^-(T)\|_{L^1}
\le \frac{r}{2}+\|w^-(0)\|_{L^1}
=1-\frac{r}{2}\le\frac{3}{4}.
\end{align*}
This coincides with the first inequality in~\eqref{L1-squeezing} with $q=\frac34$ and $\|w(0)\|_{L^1}=1$. 

We now assume that $\|w^+(T)\|_{L^1}>\frac{r}{2}=\frac12\|w_0^+\|_{L^1}$, so that inequality~\eqref{L1-squeezing} with $q=\frac12$ does not hold for the positive solution~$w^+$. Then we can find $\delta'>0$ such that 
\begin{equation} \label{lower-bound-w+}
	\inf_{x\in\Pi}w^+(T,x)\ge \delta'\,\|w_0^+\|_{L^1}\ge \frac{\delta'}{2}. 
\end{equation}
If, in addition, 
$$
\sup_{x\in\Pi}w^-(T,x)\le \frac{\delta'}{4},
$$
then inequality~\eqref{lower-bound} holds with $\delta=\frac{\delta'}{4}$. Thus, we can assume that 
\begin{equation}  \label{w-lower-bound}
	w^-(T,x_0)\ge\frac{\delta'}{4}
\end{equation}
for some point $x_0\in\Pi$. Using representation~\eqref{w-green} with $w=w^-$ and inequality~\eqref{green-bound} with $j=0$ and~$|\alpha|=1$, we derive
$$
\sup_{x\in D}|\nabla w^-(T,x)|\le M_1
$$
where $M_1$ does not depend on~$w$. It follows that 
\begin{align*}
w^-(T,x)&=w^-(T,x_0)+w^-(T,x)-w^-(T,x_0)
\ge \frac{\delta'}{4}-M_1|x-x_0|
\ge \frac{\delta'}{8},
\end{align*}
provided that $|x-x_0|\le r_0:=\delta'(8M_1)^{-1}$. Denoting by~$B$ the ball of radius~$r_0$ centred at~$x_0$ and assuming without loss of generality\footnote{We can repeat the above argument for a small subset of~$\Pi$ and find a point $x_0$ that satisfies~\eqref{w-lower-bound} and lies at some positive distance from the boundary of~$\Pi$.} that $B\subset\Pi$, we derive
\begin{equation} \label{lower-bound-w-}
	\inf_{x\in B}w^-(T,x)\ge\kappa:=\frac{\delta'}{8}. 
\end{equation}
In this case, denoting by~${\mathbf1}_B$ the indicator function of~$B$ and using inequalities~\eqref{lower-bound-w+} and~\eqref{lower-bound-w-}, we can write
\begin{align*}
	\|w(T)\|_{L^1}
	&=\int_D|w^+(T,x)-w^-(T,x)|\,\dd x\\
	&=\int_D\bigl|\bigl(w^+(T,x)-\kappa {\mathbf1}_B(x)\bigr)-\bigl(w^-(T,x)-\kappa{\mathbf1}_B(x)\bigr)\bigr|\,\dd x\\
	&\le\int_D\bigl(w^+(T,x)-\kappa {\mathbf1}_B(x)\bigr)\,\dd x
	+\int_D\bigl(w^+(T,x)-\kappa {\mathbf1}_B(x)\bigr)\,\dd x \\
	&\le \|w^+(T)\|_{L^1}+\|w^-(T)\|_{L^1}-2\kappa\,\LLL(B)\\
	&\le 1-2\kappa\,\LLL(B)=:q, 
\end{align*}
where the last inequality follows from the contraction of the $L^1$-norm of solutions for~\eqref{linear-uncontrolled}. This completes the proof of the proposition.

\subsection{Proof of Theorem~\ref{t-control}}
\label{s-control-proof}
{\it Step~1: Reduction\/}. 
Choosing $\eta\equiv0$ on the time interval $[0,1)$, we see that, in view of Corollary~\ref{c-uniform-bounds} and Proposition~\ref{p-L1-contraction}, 
\begin{gather*}
	\|u(1)\|_{C^{2+\gamma}}+\|\hat u(1)\|_{C^{2+\gamma}}\le C_1,\\
	\|u(1)-\hat u(1)\|_{L^1}\le \|u_0-\hat u(0)\|_{L^1}, 
\end{gather*}
where $C_1\ge1$ is a number not depending on~$u_0$ and~$\hat u$. Therefore, we can assume from the very beginning that the $C^{2+\gamma}$ norms of~$u_0$ and~$\hat u(0)$ are bounded by some universal constant and prove inequality~\eqref{expo-stab} for $t\ge0$. 

Suppose we found a control $\eta(t,x)=\xi(t)\varphi(x)$, with a piece-wise constant function~$\xi$ satisfying~\eqref{control-estimate} and~\eqref{control-structure}, such that, for any $t\ge0$, 
\begin{align}
	\|u(t)-\hat u(t)\|_{C^{2+\gamma}}&\le C_2,\label{C2gamma-bound}\\
	\|u(t)-\hat u(t)\|_{L^1}&\le C_2e^{-at}\|u_0-\hat u(0)\|_{L^1},\label{L1-expodecay}
\end{align}
 where $C_2$ and~$a$ are positive numbers not depending on~$u_0$ and~$\hat u$. In this case, by the interpolation inequality~\eqref{GNS-inequality}, for any $\sigma\in(0,\gamma)$ we can write
$$
\|u(t)-\hat u(t)\|_{C^{2+\sigma}}\le C\,C_2 \bigl(e^{-at}\|u_0-\hat u(0)\|_{L^1}\bigr)^\theta, 
$$
where $\theta=\frac{\gamma-\sigma}{d+\gamma+2}$. We thus obtain inequality~\eqref{expo-stab} with $\alpha=\theta a$ and $\beta=\theta$.

\medskip
{\it Step~2: Uniform bound of the $C^{2+\gamma}$-norm\/}. We claim that if~$\xi$ has the form described in the theorem, then inequality~\eqref{C2gamma-bound} holds with a sufficiently large constant~$C_2$ depending only on~$|\!|\!|h|\!|\!|_\gamma$. Indeed, a straightforward consequence of Corollary~\ref{c-uniform-bounds} is that $\|\hat u(t)\|_{C^{2+\gamma}}\le C_3$ for all $t\ge0$ and 
\begin{equation} \label{C2gamma-bound-integer}
\|u(k)\|_{C^{2+\gamma}}\le C_3 
\end{equation}
for all integers $k\ge0$, where we write~$C_i$ for unessential positive numbers that may depend only on~$|\!|\!|h|\!|\!|_\gamma$. Thus, the required bound~\eqref{C2gamma-bound} will be established if we prove that, for any $k\ge0$, inequality~\eqref{C2gamma-bound-integer} implies that 
\begin{equation} \label{C2gamma-bound-interval}
\|u(t)\|_{C^{2+\gamma}}\le C_4 \quad\mbox{for $t\in [k,k+1]$}. 
\end{equation}
 In view of translation invariance, there is no loss of generality in assuming that $k=0$. We thus consider the solution~$u(t,x)$ of problem~\eqref{controlled-PDE}, \eqref{BC}, \eqref{IC} on the time interval $J=[0,1]$, with an initial condition~$u_0\in C^{2+\gamma}(D)$ and a control function of the form
 \begin{equation} \label{eta}
 \eta(t,x)=\xi(t)\varphi(x), \quad \xi(t)=\left\{
 		\begin{array}{cl}
			0  & \mbox{for $t\in\bigl[0,\tfrac12\bigr)\cup [T_\tau,1]$},\\[3pt]
			\bar\xi & \mbox{for $t\in \bigl[\tfrac12,T_\tau\bigr)$},
		\end{array}
 \right.  	
 \end{equation}
 where $T_\tau=\frac12+\tau$, $\tau\in(0,\frac14)$, and $\tau|\bar\xi|\le C_5$. The function~$u$ is constructed by applying Remark~\ref{r-existence} in the domain $Q=J\times D$. This  results in a continuous curve $u:J\to C^{2+\gamma}(D)$ whose restrictions to the  cylinders 
$$
\Omega_1:=(0,\tfrac12)\times D, \quad \Omega_2:=(\tfrac12,T_\tau)\times D, \quad \Omega_3=(T_\tau,1)\times D
$$
belong to the spaces $\XX^\gamma(\Omega_i)$, $i=1,2,3$. By the maximum principle applied consecutively to~$u|_{\Omega_i}$ (see Proposition~\ref{p-maximum}), the $L^\infty$ norm of~$u$ is bounded by a universal constant (i.e., a number not depending on~$\tau$ and~$\bar\xi$). Setting $g=u-\zeta$ with $\zeta(t,x)=\int_0^t\eta(s,x)\,\dd s$, we see that $g(t)$ is a continuous $C^{2+\gamma}(D)$-valued function satisfying the equation
\begin{equation*}
	\p_tg-\nu\Delta g+\sum_{j=1}^d \p_j A_j(g+\zeta(t,x))=h+\nu\Delta\zeta.
\end{equation*}
Rewriting this as the inhomogeneous heat equation 
\begin{equation*} \label{heat}
\p_tg-\nu\Delta g=f(t,x):=h+\nu\Delta\zeta-\sum_{j=1}^d \p_j A_j(g+\zeta(t,x))	
\end{equation*}
and applying Theorem~10.1 in~\cite[Chapter~III]{LSU1968}, we conclude that, for some $\gamma'\in(0,\gamma]$ the norm $|\!|\!|g|\!|\!|_{\gamma'}$ is bounded by a universal constant. A simple calculation shows that if 
\begin{equation} \label{tau-bound}
\tau^{1-\gamma}|\bar\xi|\le 1,
\end{equation}
then for any multi-index $\alpha\in\Z_+^d$ there is a universal constant~$C_\alpha$ such that 
\begin{equation} \label{bound-zeta}
	|\!|\!|\p_x^\alpha \zeta|\!|\!|_\gamma\le C_\alpha. 
\end{equation}
Using now Theorem~5.2 in~\cite[Chapter~IV]{LSU1968} with $l=\gamma'$, we obtain a universal bound for $\|g\|_{\XX^{\gamma'}(Q)}$. Applying again the same result with~$l=\gamma$, we derive a universal bound for $\|g\|_{\XX^{\gamma}(Q)}$. The required bound~\eqref{C2gamma-bound-interval} follows from the representation $u=g+\zeta$ and inequality~\eqref{bound-zeta}. 

\medskip
{\it Step~3: Decay of the $L^1$-norm\/}. We now prove~\eqref{L1-expodecay}. To this end, suppose we have established the following estimate for all integers $k\ge1$:
\begin{equation} \label{L1-expodecay-integer}
	\|u(k)-\hat u(k)\|_{L^1}\le q_1\|u(k-1)-\hat u(k-1)\|_{L^1}. 
\end{equation} 
In this case, inequality~\eqref{L1-expodecay} will follow immediately if we prove that 
\begin{equation} \label{L1-bound}
	\|u(t)-\hat u(t)\|_{L^1}\le C\,\|u(k)-\hat u(k)\|_{L^1}
	\quad\mbox{for $t\in[k,k+1]$}.
\end{equation}
By homogeneity, it suffices to consider the case $k=0$. The difference $w=u-\hat u$ is  the solution of the problem
\begin{equation} \label{difference-w}
	\p_tw-\nu\Delta w+\diver_x(a(t,x)w)=\eta(t,x),\quad w\big|_{\p D}=0, \quad w(0)=w_0,
\end{equation}
where $w_0=u(0)-\hat u(0)$, $\eta$ is defined by~\eqref{eta}, and the vector function~$a$ is given by~\eqref{function-a}. Let $\psi:\R\to\R$ be a smooth non-decreasing function vanishing at zero and such that $\psi(s)=\pm1$ for $\pm s\ge1$, and let $\psi_\e(s)=\psi(s/\e)$. We multiply the first relation in~\eqref{difference-w} by~$\psi_\e(w)$ and integrate over~$D$. Denoting by $\Psi_\e$ the primitive of~$\psi_\e$ vanishing at zero, after some simple transformations, we derive
$$
\p_t\int_D\Psi_\e(w)\dd x+\frac{\nu}{\e}\int_D|\nabla w|^2\psi'(\tfrac{w}{\e})\dd x=
\frac{1}{\e}\int_Dw\langle a,\nabla w\rangle \psi'(\tfrac{w}{\e})\dd x+\int_D\eta \psi(\tfrac{w}{\e})\dd x. 
$$
Using the boundedness of~$\psi$ and the Cauchy--Schwarz inequality,  we obtain
\begin{equation} \label{psi-theta}
	\p_t\int_D\Psi_\e(w)\dd x
	\le C_\nu \int_D\tfrac{|w|^2}{\e}\,\psi'(\tfrac{w}{\e})\,\dd x 
	+\xi(t)\|\varphi\|_{L^1}.
\end{equation}
Since the integral on the right-hand side is bounded by $\|w\|_{L^1}$, integrating~\eqref{psi-theta} in time and passing to the limit as $\e\to0^+$, we obtain
$$
\|w(t)\|_{L^1}\le C_\nu\int_0^t\|w(s)\|_{L^1}\dd s+d+\tau|\bar\xi|\,\|\varphi\|_{L^1},
$$
where $d=\|u(0)-\hat u(0)\|_{L^1}$. Assuming that 
\begin{equation} \label{tau-xi-bound}
	\tau|\bar\xi|\le K_1d, 
\end{equation}
where $K_1>0$ is fixed,  and applying Gronwall's inequality, we arrive at~\eqref{L1-bound} with $k=0$. 

It remains to prove~\eqref{L1-expodecay-integer}. There is no loss of generality in assuming that $k=0$, so that we prove~\eqref{squeezing-final}. We shall need the following auxiliary result. 

\begin{lemma} \label{l-comparison}
Let $\tau,R>0$ and $\bar\xi\in\R$ be some numbers such that $\tau|\bar\xi|\le1$, let $\eta(t,x)=\bar\xi\varphi(x)$ for $0\le t\le \tau$, and let $u,\hat u\in C^{1,2}(\overline Q_\tau)$ be solutions of Eqs.~\eqref{controlled-PDE}  and~\eqref{PDE} respectively such that $\|u\|_{C^{1,2}}+\|\hat u\|_{C^{1,2}}\le R$. Then there is $C>0$ depending only on~$R$ such that the difference $z=u-\hat u$ satisfies the inequality 
	\begin{equation} \label{comparison-estimate}
		\bigl\|z(t)-(z(0)+t\,\bar\xi\varphi)\bigr\|_{L^\infty(D)}\le C\,t\quad
		\mbox{for $0\le t\le \tau$}. 
	\end{equation}
\end{lemma}

Taking this lemma for granted, let us complete the proof of~\eqref{L1-expodecay-integer}. As was established at the end of Section~\ref{s-control-result}, the required inequality~\eqref{squeezing-final} will be proved if we show that~\eqref{difference-z-Z} holds, and the restrictions imposed on~$\tau$ and~$\bar\xi$ are compatible. The validity of~\eqref{difference-z-Z} follows from inequality~\eqref{comparison-estimate} rewritten on the interval~$[\frac12,\frac12+\tau]$. To see the compatibility of restrictions, let us recall that we seek $\theta>0$ such that the numbers 
$$
\tau=(2C_1M)^{-1}\theta\delta d\,\|\varphi\|_{L^1}, \quad |\bar\xi|=(\tau M)^{-1}\theta\delta d=2C_1\|\varphi\|_{L^1}^{-1}
$$
satisfy inequalities~\eqref{tau-bound} and~\eqref{tau-xi-bound}. Since~$d$ is bounded above by a universal number and $\gamma\in(0,1)$, it is easy to see that both~\eqref{tau-bound} and~\eqref{tau-xi-bound} hold for $\theta\ll1$, and relations~\eqref{control-estimate} are valid for $K=2C_1\|\varphi\|_{L^1}^{-1}$ and $\varkappa=(2C_1M)^{-1}\theta\delta\,\|\varphi\|_{L^1}$.  We have thus completed the proof of Theorem~\ref{t-control}. 

\begin{proof}[Proof of Lemma~\ref{l-comparison}]
	The difference $z=u-\hat u$ belongs to the space~$C^{1,2}(\overline Q_\tau)$ and satisfies Eqs.~\eqref{z-problem}, in which the initial time $t=\frac12$ is replaced by~$t=0$. Setting 
	$$
	f(t,x)=\nu\Delta z-\diver\bigl(A(u)-A(\hat u)\bigr),
	$$
we see that the function $z_1(t)=z(t)-z(0)-t\bar\xi\varphi$ is a solution of the problem
	\begin{equation} \label{z1-equation}
	\p_tz_1=f(t,x), \quad z_1(0)=0.		
	\end{equation}
	The hypotheses imposed on~$u$ and~$\hat u$ imply that the $L^\infty$ norm of~$f$ is bounded by a constant depending only on~$R$. Integrating~\eqref{z1-equation} in time, we arrive at the required estimate~\eqref{comparison-estimate}. 
\end{proof}

\section{Appendix}
\label{s-appendix}

\subsection{Green's function for linear parabolic equations}
\label{ss-green}
Let $D\subset\R^d$ be a bounded domain with $C^3$-smooth boundary, let $J_T=[0,T]$ be an interval, and let~$G(t,x,y)$ be the Green function for the second-order parabolic PDE
\begin{equation} \label{parabolic-PDE}
	\p_t w-\nu\Delta_x w+\langle b(t,x),\nabla\rangle w+b_0(t,x)w=0, \quad x\in D,
\end{equation}
where the functions $b=(b_1,\dots,b_d)$ and~$b_0$ belong to the H\"older class $C^{\frac{\gamma}{2},\gamma}(Q_T)$ for some $\gamma\in(0,1)$. In other words, $G$ is the solution of~\eqref{parabolic-PDE} satisfying the initial and boundary conditons
$$
G\bigr|_{t=0}=\delta_y, \quad
G\bigr|_{x\in \p D}=0,
$$
where $\delta_y$ stands for the Dirac mass at~$y$. A proof of the existence of~$G$, as well as the following inequalities, can be found in~\cite{LSU1968} (see Theorems~16.2 and~16.3 in Chapter~IV):
\begin{align}
|\p_t^j\p_x^\alpha G(t,x,y)|&\le 
C\,t ^{-\frac{d+2j+|\alpha|}{2}}e^{-c\frac{|x-y|^2}{t}},
\label{green-bound}\\
|\p_t^j\p_x^\alpha G(t,x_1,y)-\p_t^j\p_x^\alpha G(t,x_2,y)|&\le 
C\,|x_1-x_2|^\gamma t ^{-\frac{d+2+|\alpha|}{2}}e^{-c\frac{d(x_1,x_2,y)^2}{t}},
\label{green-holder}
\end{align}
where $0<t\le T$, $x,x_1,x_2,y\in D$, $d(x_1,x_2,y)=|x_1-y|\wedge|x_2-y|$, $j\ge0$ and $\alpha\in \Z_+^d$ are such that $2j+|\alpha|\le 2$, and~the constants~$C$ and~$c$ on the right-hand side depend only $\nu$, $T$, $D$, $\gamma$, and the norms of the functions~$b$ and~$b_0$ in the space~$C^{\frac{\gamma}{2},\gamma}(Q_T)$. 

\subsection{Harnack's inequality and $L^1$-contraction}
We first recall a fundamental estimate for positive solutions of~\eqref{parabolic-PDE}. The following proposition is a simple particular case of more general results established in~\cite{moser-1964,landis1998,KS-1980}.

\begin{proposition} \label{p-Harnack}
Let $T>0$ and let $b_0,b\in L^\infty(Q_T)$. Then, for any compact set $K\subset D$ and any numbers $0<s<t\le T$, there is $C>0$ depending on $\|b_0\|_{L^\infty(Q_t)}+\|b\|_{L^\infty(Q_t)}$ such that, if $w\in L_{\mathrm{loc}}^\infty(\R_+\times D)$ is a non-negative solution of~\eqref{parabolic-PDE} satisfying the Dirichlet boundary condition, then
\begin{equation} \label{harnack}
\sup_{x\in K}w(s,x)\le C\inf_{x\in K}w(t,x).
\end{equation}
\end{proposition}

Let us note that~\cite{moser-1964,landis1998} deal with parabolic equations with no lower-order terms, while the paper~\cite{KS-1980} (see Theorem~1.1) considers the general case, assuming that the coefficient in front of the zeroth-order term satisfies a sign condition. That condition can easily be removed by the change of unknown function $u=e^{\lambda t}v$, with a suitable choice of~$\lambda\in\R$. Moreover, it follows from~\eqref{green-bound} that any solution $w\in L_{\mathrm{loc}}^\infty(\R_+\times D)$ for Eq.~\eqref{parabolic-PDE} with the Dirichlet boundary condition is a continuous function of its arguments for~$t>0$, so that the left- and right-hand sides of~\eqref{harnack} are well defined. 

We now turn to the particular case of~\eqref{parabolic-PDE}. Namely, consider the linear problem~\eqref{linear-uncontrolled}, in which $a\in C(\overline Q_T)$ is a given function. The following result is a simple consequence of the maximum principle and a duality argument, and its proof can be found in~\cite[Lemma~3.2.2]{hormander1997}.

\begin{proposition} \label{p-L1-contraction}
	Let $w\in C^{1,2}(\overline Q_T)$ be a solution of problem~\eqref{linear-uncontrolled} in which $a\in C(\overline Q_T)$. Then $\|w(T)\|_{L^1}\le \|w(0)\|_{L^1}$.
\end{proposition}

\subsection{Maximum principle for the nonlinear problem}

\begin{proposition} \label{p-maximum}
	Let $u\in C^{1,2}(\overline Q_T)$ be a solution of~\eqref{PDE} with $A\in C^1(\R,\R^d)$ and $h\in C(\overline Q_T)$. Then 
	\begin{equation} \label{MP-nonlinear}
		\|u(t)\|_{L^\infty(Q_t)}
		\le \|u(0)\|_{L^\infty(D)}+\int_0^t\|h(s)\|_{L^\infty(D)}\dd s,\quad 0\le t\le T.
	\end{equation}
\end{proposition}

\begin{proof}
	Since the function $-u$ satisfies an equation of the same form, it suffices to prove that $u(t,x)$ does not exceed the right-hand side of~\eqref{MP-nonlinear} for $(t,x)\in Q_T$. To this end, we introduce the functions
	$$
	H(t)=\int_0^t\|h(s)\|_{L^\infty(D)}\dd s,\quad a_j(t,x)=A_j'(u(t,x))
	$$
	and note that $v:=u-H\in C^{1,2}(\overline{Q}_T)$ satisfies the differential inequality
	$$
	\p_tv+\sum_{j=1}^da_j(t,x)\p_jv-\nu\Delta v\le 0.
	$$
	By the maximum principle, we have
	$$
	v(t,x)\le \max(u(0),0)\le \|u(0)\|_{L^\infty(D)}, \quad (t,x)\in Q_T.
	$$
Combining this with the definition of~$v$, we obtain the required inequality. 
\end{proof}

\addcontentsline{toc}{section}{Bibliography}
\def\cprime{$'$} \def\cprime{$'$} \def\cprime{$'$}
  \def\polhk#1{\setbox0=\hbox{#1}{\ooalign{\hidewidth
  \lower1.5ex\hbox{`}\hidewidth\crcr\unhbox0}}}
  \def\polhk#1{\setbox0=\hbox{#1}{\ooalign{\hidewidth
  \lower1.5ex\hbox{`}\hidewidth\crcr\unhbox0}}}
  \def\polhk#1{\setbox0=\hbox{#1}{\ooalign{\hidewidth
  \lower1.5ex\hbox{`}\hidewidth\crcr\unhbox0}}} \def\cprime{$'$}
  \def\polhk#1{\setbox0=\hbox{#1}{\ooalign{\hidewidth
  \lower1.5ex\hbox{`}\hidewidth\crcr\unhbox0}}} \def\cprime{$'$}
  \def\cprime{$'$} \def\cprime{$'$} \def\cprime{$'$}
\providecommand{\bysame}{\leavevmode\hbox to3em{\hrulefill}\thinspace}
\providecommand{\MR}{\relax\ifhmode\unskip\space\fi MR }
\providecommand{\MRhref}[2]{%
  \href{http://www.ams.org/mathscinet-getitem?mr=#1}{#2}
}
\providecommand{\href}[2]{#2}

\end{document}